\documentclass[12pt]{article}
\usepackage{subfigure}
\usepackage{booktabs}   
\usepackage{mathrsfs}
\usepackage{amsfonts}
\usepackage{amsmath}
\usepackage{amssymb}
\usepackage{amsthm}
\usepackage{enumerate}
\usepackage[mathscr]{eucal}
\usepackage{eqlist}
\usepackage{graphicx}
\usepackage{graphicx, color}
\usepackage{indentfirst}
\usepackage[english]{babel}
\usepackage{tikz}
\usepackage[numbers,sort&compress]{natbib}
\usepackage{epstopdf}

\usepackage[body={16.4 true cm, 21.0true cm}]{geometry}

\theoremstyle{plain}
\newtheorem{theorem}{Theorem}[section]

\newtheorem{lemma}[theorem]{Lemma}

\theoremstyle{definition}
\newtheorem{definition}[theorem]{Definition}

\newtheorem{remark}[theorem]{Remark}

\theoremstyle{remark}
\newtheorem{example}[theorem]{Example}
\numberwithin{equation}{section}

\begin{document}
\title{ A three-term Polak-Ribi\`{e}re-Polyak conjugate gradient method for vector optimization
\footnote{\ This work is supported by National Nature Science Foundation of China (No. 12371304).} }

\author{Guangxuan Lin \footnote{\,
School of Mathematics and Statistics, Qingdao University, Qingdao 266071, Shandong Province, P.R. China. }
\ \ and \ \
Shouqiang Du \footnotemark[2]~ \footnote{\ Corresponding author: Shouqiang Du.  E-mail address: sqdu@qdu.edu.cn.}}
\date{}\maketitle
\begin{abstract}
A novel three-term Polak-Ribi\`{e}re-Polyak conjugate gradient method is  proposed for solving  vector optimization problems. It should be emphasized  that this  is  the first extension of   three-term  conjugate gradient methods  from scalar optimization   to vector optimization.
The method can consistently generate a sufficient descent direction independent of line search procedures and without modifying the conjugate parameters. This result improves upon the corresponding conclusions in SIAM J. Optim. \textbf{28}, 2690-2720 (2018), J. Optim. Theory Appl. \textbf{204},
13 (2025) and Optim. Methods Softw. \textbf{28}, 725-754 (2025).
 Based on a new Wolfe-type line search, the  global convergence of the proposed scheme is  established   without imposing restrictions such as self-adjusting strategies, regular restarts and convexity assumptions. Numerical experiments demonstrate  the favourable performance of the proposed method.

\end{abstract}

\medskip
	

\maketitle

\section{Introduction}
Conjugate gradient (CG) method, known for its fast convergence  and low memory requirements, is an efficient  first-order algorithm for solving the unconstrained optimization problem
\begin{equation*}
\min F(x),\qquad \forall x\in\mathbf{R}^n,
\end{equation*}
with the iteration
$$x^{k+1}=x^k+\alpha_kd^k\qquad\textrm{for}~k=0,1,2,\ldots,$$
where $F:\mathbf{R}^n\to\mathbf{R}$ is continuously differentiable, $\alpha_k>0$ stands for  the step size and $d^k$ denotes the  search direction given by
$$d^k=\left\{\begin{array}{l}-\nabla F(x^k)\qquad\text{for}\ k=0,\\
-\nabla F(x^k)+\beta_kd^{k-1}\qquad\text{for}\ k\ge1.
\end{array}\right.$$
Here, $\beta_k$ is called the conjugate  parameter.  Different choices of conjugate parameters correspond to  various CG methods. Several famous CG methods are listed as follows:
$$\begin{aligned}
&\text{Dai-Yuan (DY)} \cite{Dai1}:~~\beta_k^{DY}=-\frac{\langle\nabla F(x^k), \nabla F(x^k)\rangle}{\langle d^{k-1}, \nabla F(x^{k})-\nabla F(x^{k-1})\rangle};\\[1mm]
&\text{Fletcher-Reeves (FR)} \cite{Fle}:~~\beta^{FR}_k=\frac{\langle\nabla F(x^k), \nabla F(x^k)\rangle}{\langle\nabla F(x^{k-1}), \nabla F(x^{k-1})\rangle};\\[1mm]
&\text{Hestenes-Stiefel (HS)} \cite{Hes}:~~\beta^{HS}_k=\frac{\langle\nabla F(x^k), \nabla F(x^k)-\nabla F(x^{k-1})\rangle}{\langle d^{k-1},\nabla F(x^k)-\nabla F(x^{k-1})\rangle};\\[1mm]
&\text{Polak-Ribi\`{e}re-Polyak (PRP)} \cite{Pol}:~~\beta^{PRP}_k=\frac{\langle\nabla F(x^k), \nabla F(x^k)-\nabla F(x^{k-1})\rangle}{\langle\nabla F(x^{k-1}), \nabla F(x^{k-1})\rangle}.
\end{aligned}$$

Usually, $d^k$ should satisfy $\langle\nabla F(x^k),d^k\rangle< 0$ to ensure the  well-definedness of CG methods. In most convergence analyses, we  need a stronger descent condition on $d^k$, called sufficient descent condition, given as
$$\langle\nabla F(x^k),d^k\rangle\leq h\|\nabla F(x^k)\|^2$$
for some $h>0$.

It is well known that the PRP  CG method shows  better performance compared with other CG methods  in  real world applications. The global convergence of the PRP CG method was established in \cite{Pol} by using the exact line search procedure under the strong convexity hypothesis of the objective function. However, when the objective function is nonconvex,  the exact line search can lead to a nonconvergent sequence based on the PRP CG method (see \cite{Pow}). In \cite{Gil}, Gilbert and  Nocedal showed the global convergence of the PRP CG method under the condition that the conjugate parameters are restricted to be nonnegative. Dai et al. \cite{Dai} claimed   that the positiveness constraint of conjugate parameters cannot be relaxed for the PRP CG method.

On the other hand, some variations of CG methods, known as the three-term CG method, have been widely studied, such as the three-term FR CG method \cite{Zha2}, the three-term HS CG method  \cite{Zha3}
and the three-term PRP CG method \cite{Zha1,And}.  The general form of  three-term CG methods was proposed in  \cite{Nar,Al}.
An attractive property of these  three-term CG methods is that the search direction generated by these methods is invariably a sufficient descent direction, regardless of any line search and without modifying the conjugate parameters.

Vector optimization,  which originates  from the practical demands of multi-criteria decision making,  is widely applied in modern scientific fields such as financial investment \cite{Ali,Flie}, engineering \cite{De}, resource allocation \cite{Gra},  machine design \cite{Jah} and space exploration \cite{Tav}. Very recently, the theory of vector optimization has  found increasing application  in frontier domains including machine learning \cite{ZhouS,SunJ}, neural networks \cite{WangZ}, computer engineering \cite{XuY} and sensor technology \cite{JiangW}.
Consequently, the study of vector optimization problems has increasingly become  essential.

Unconstrained vector optimization problem aims to minimize or maximize  a vector-valued function in the setting of   a partial order induced by a cone in the image space.
Given a continuously differentiable function  $\Phi:\mathbf{R}^n\to\mathbf{R}^m$ and a closed, convex and pointed cone $ E\subset \mathbf{R}^m$ with int$(E)\neq\emptyset$,  the  problem is usually defined as
\begin{equation*}\label{1}
\min_E~\Phi(u),\qquad \forall u\in\mathbf{R}^n.
\end{equation*}
Here, the minimum concept is defined relative to the partial order  induced by the cone $E$, called  Pareto optimality  (see Section 2 below). Moreover, when $E=\mathbf{R}^m_+$ with the usual induced partial order, this problem can be reduced to the multiobjective  optimization problem, which seeks to  find acceptable trade-offs between the objectives.
In this setting, a point is called Pareto optimal if  no other point can improve  one or more objectives without worsening any other objective.
The  subject  of multiobjective optimization can be found in \cite{Elb,Fli,Fli2,Hub,Liu,Pen}.

Recently, some first-order  and second-order scalar  methods, including but not limited to the steepest descent method \cite{Dru},  the subgradient method \cite{Bel}, the projected gradient method \cite{Dru2}, the Newton-type method \cite{Chu,Dru1} and the quasi-Newton method \cite{Ans}, have been  extended to  vector optimization.  Unlike the steepest descent method, which often exhibits the sawtooth phenomenon, the CG method achieves faster convergence by  using the conjugate search direction, avoiding inefficient  and redundant  iteration steps.  Compared to the Newton-type method, the CG method is computationally more efficient as it avoids computing the Hessian matrix of the objective function and its inverse.  Furthermore, the CG method can be compared with the Newton-type method in terms of  convergence rate.
Hence,  exploring the vector extension of CG methods  holds great promise.

More recently, in \cite{Luc}, the  FR, CD, DY, PRP and HS CG methods were first extended to solve the unconstrained vector optimization  problem. Subsequently, the vector extensions of the Hager-Zhang  CG method \cite{Gon2}, the Liu-Storey CG method \cite{Gon1}, the  Dai-Liao CG method \cite{HuQ}, the spectral CG method \cite{He} and the hybrid CG method \cite{Yaha}  were established.  More research on CG methods for vector optimization was considered in \cite{ChenK,Hu,Ya,Zhan}. It is important to note that in these works, the  vector extension of CG methods  may not always yield a descent direction, even when employing the exact or Wolfe-type line search procedure. This phenomenon is prevalent in vector optimization, as illustrated by the examples provided in \cite{Gon2,Gon1,He,HuQ}. To ensure the descent property of CG methods in these studies,  it is necessary to  either rely on sufficiently accurate line search procedures (see e.g., \cite{Gon2,Gon1}) or redefine the conjugate parameters (see  e.g., \cite{He,Hu,Ya,Zhan,Chen}).
However, adopting a specific line search  or modifying the conjugate parameters may cause additional computational costs, as shown by the numerical experiments in \cite{Gon2,Gon1,Hu}.
Therefore, it is valuable to explore a descent CG method for vector optimization regardless of  any line search and without the modification of conjugate parameters.

On the other hand, the vector extension of modified PRP CG methods  has been discussed in \cite{Chen,HuQ1} recently. However, as far as we know,  the descent  PRP-type CG method,  independent of any line search and  without modifying the conjugate parameters, has not been extended to vector optimization, which is one of the main motivations for the current study.

Motivated by the above works,   a new vector extension of the three-term PRP CG method is  proposed in this paper.
We show the  sufficient descent property and the  convergence result of the proposed scheme.
Numerical experiments are also presented to illustrate  the  excellent performance of the  method.
The main novelties of this paper are as follows:
\begin{itemize}
\item This work can be regarded as the first extension of three-term CG methods from scalar optimization  to vector optimization, which fills  a key  gap in the application of  three-term CG methods  in the field of vector optimization.

\item The proposed extension  has the excellent characteristics of the traditional three-term PRP CG method, i.e., it can  always  generate a  sufficient descent direction  irrespective of any line search and without redefining the conjugate parameters, which improves upon the corresponding results in \cite{Luc,Chen,HuQ1}.

\item  Based on a novel line search, called  the generalized Wolfe line search,  we establish the global convergence of the proposed scheme  without relying on restrictive conditions such as self-adjusting strategies, regular restarts and convexity assumptions.
\end{itemize}

The  rest of the paper  is structured as follows:
In  Sect. 2, some  notions and preliminary results related to vector optimization  are presented. In Sect. 3, we propose a vector extension of the three-term PRP CG method that invariably  generates  a sufficient descent direction independent of any line search and without modifying the conjugate parameters.  In Sect. 4, the convergence result of the proposed method is established  using a generalized Wolfe line search procedure. In Sect. 5, we provide some numerical experiments  to show the favourable behavior of the scheme. Some remarks and future directions are discussed in the last section.

\section{Preliminaries}
Let $ E\subset \mathbf{R}^m$ be a closed, convex and pointed cone with int$(E)\neq\emptyset$. Define the partial order $\preceq_E$ in $\mathbf{R}^m$ induced by $E$ (resp. $\prec_E$ in $\mathbf{R}^m$ induced by int$(E)$)   by
$$ u_1\preceq_E u_2\Leftrightarrow u_2-u_1\in E~(\text{resp}.~ u_1\prec_E u_2\Leftrightarrow u_2-u_1\in\text{int}(E)).$$

In this work, we mainly  focus on  the unconstrained vector optimization problem
\begin{equation}\label{1}
\min_E~\Phi(u),\qquad \forall u\in\mathbf{R}^n,
\end{equation}
where $\Phi:\mathbf{R}^n\to\mathbf{R}^m$ is continuously differentiable.
We call  $u^*\in\mathbf{R}^n$   $E$-Pareto optimal   iff there is no $u\in\mathbf{R}^n$ such that $\Phi(u)\preceq_E\Phi(u^*)$ and $\Phi(u)\neq \Phi(u^*)$. For $u\in\mathbf{R}^n$, we call it $E$-Pareto critical for $\Phi$ if it satisfies the condition
\begin{equation*}
-\text{int}(E)\cap\textrm{Im}(J\Phi(u))=\emptyset,
\end{equation*}
where $J\Phi(u)$ denotes the Jacobian of $\Phi$ at $u$ and Im$(J\Phi(u))$ stands for the image on $\mathbf{R}^n$ by $J\Phi(u)$. This condition is  also a necessary condition for  local Pareto optimality.

If $u\in\mathbf{R}^n$ is not $E$-Pareto critical, then there exists $d\in\mathbf{R}^n$ such that
$ J\Phi(u)d\in-\text{int}(E)$, meaning that
\begin{equation}\label{17-}
 \Phi(u+\kappa d)\prec_E\Phi(x)
\end{equation}
for all $\kappa\in(0,\varepsilon)$ with some $\varepsilon>0$.  We call $d\in\mathbf{R}^n$  an $E$-descent direction for $\Phi$ at $u$ if it satisfies (\ref{17-}).

We next proceed to describe descent directions and critical points  for vector optimization with scalar results.
Denote  the  dual polar cone of $E$ as
$$E^*:=\{ v\in \mathbf{R}^m: \langle v,z\rangle\geq0,\ \forall z\in E\}. $$
Since $E$ is closed and convex, we use the definition of $E^*$ to obtain $E=E^{**}$,
$$-E=\{v\in\mathbf{R}^m: \langle v,z\rangle\leq0,\ \forall z\in E^*\}$$
and
$$-\text{int}(E)=\{v\in\mathbf{R}^m:\langle v,z\rangle<0,\ \forall z\in E^*-\{0\}\}. $$
Let $V\subset E^*-\{0\}$ be a compact set  such that
$$\text{cone}(\text{co}(V))=E^*,$$
where co$(V)$ denotes the convex hull of $V$ and cone$(\text{co}(V))$ stands for the cone generated by co$(V)$.
By the definition of $V$, we obtain $0\notin$ co$(V)$. Hence,
$$-E=\{v\in\mathbf{R}^m: \langle v,z\rangle\leq0,\ \forall z\in V\}$$
and
$$-\text{int}(E)=\{v\in\mathbf{R}^m: \langle v,z\rangle<0,\ \forall z\in V\}.$$
For a general  cone $E$, the set
$$ V:=\{v\in E^*:\|v\|=1\}$$
has the above property.  Unless otherwise specified, $V$ will denote exactly this set.
\begin{remark}
In multiobjective optimization, $E=E^*=\mathbf{R}_+^{m}$ and $V$ is usually  taken as the canonical basis of $\mathbf{R}^m$.
\end{remark}

We now introduce a scalar function to describe  descent directions and critical points. Define $\lambda:\mathbf{R}^n\times \mathbf{R}^n\to\mathbf{R}$ by
$$\lambda(u,d):=\sup\{\langle J\Phi(u)d,v\rangle:v\in V\}.$$
Obviously, the function $\lambda$ is well-defined due to the compactness of $V$. Based on the above foundations, we can easily check that
\begin{itemize}
\item  $u$ is $E$-Pareto critical for $\Phi$ iff $\lambda(u,d)\geq0$ for all $d\in \mathbf{R}^n$;
\item  $d$ is an $E$-descent direction for $\Phi$ at $u$ iff $\lambda(u,d)<0$.
\end{itemize}
Some basic properties of the function $\lambda$ are presented as follows:
\begin{lemma}\cite{Gon1}\label{L1}
Under the function $\Phi:\mathbf{R}^n\to \mathbf{R}^m$ is continuously differentiable, we have that
\begin{enumerate}
\item[\rm(i)] for any $u,v,z\in\mathbf{R}^n$ and $\kappa>0$,
$$ \lambda(u,v+\kappa z)\leq\lambda(u,v)+\kappa\lambda(u,z);$$
\item[\rm(ii)] the function $\lambda(\cdot,\cdot)$ is continuous;
\item[\rm(iii)]  $\lambda(\cdot,d)$ is $L\|d\|$-Lipschitz continuous  if the Jacobian $J\Phi$ is $L$-Lipschitz continuous.
\end{enumerate}
\end{lemma}
Define the steepest direction $\vartheta:\mathbf{R}^n\to\mathbf{R}^n$ for  vector optimization, as proposed in \cite{Dru}, by
\begin{equation}\label{2}
\vartheta(u):={\arg\min}_{d\in\mathbf{R}^n}\lambda(u,d)+\frac{\|d\|^2}{2}.
\end{equation}
The corresponding optimal value  in (\ref{2}) is defined as
$$\Theta(u):=\lambda(u,\vartheta(u))+\frac{\|\vartheta(u)\|^2}{2}.$$
Since the function $\lambda(u,\cdot)$ is  closed and convex, $\vartheta(u)$   exists uniquely. The functions  $\vartheta$ and $\Theta$ have the following useful properties.
\begin{lemma}\cite{Dru}\label{L2}
Under the function $\Phi:\mathbf{R}^n\to \mathbf{R}^m$ is continuously differentiable, we have that
\begin{enumerate}
\item[\rm(i)] if $u$ is $E$-Pareto critical, then $\vartheta(u)=0$ and $\Theta(u)=0$;
\item[\rm(ii)] if $u$ is not $E$-Pareto critical, then $\vartheta(u)\neq0$  is an $E$-descent direction for $\Phi$ at $u$, $\Theta(u)<0$ and
$$\lambda(u,\vartheta(u))<-\frac{\|\vartheta(u)\|^2}{2}<0;$$
\item[\rm(iii)] the functions $\vartheta$ and $\Theta$ are continuous.
\end{enumerate}
\end{lemma}
\begin{remark}
In multiobjective optimization, we choose $E=\mathbf{R}^m_+$ and $V$ as the canonical basis of $\mathbf{R}^m$, then $\vartheta(u)$ can be obtained by solving the  convex quadratic problem:
\begin{equation}\label{19-}\begin{aligned}
&\min~ \alpha+\frac{\|d\|^2}{2},\\
&\text{s. t.}~~[J\Phi(u)d]_i\leq \alpha,~~~~i=1,\ldots,m.
\end{aligned}\end{equation}
\end{remark}
We now turn attention to some line search methods  for  vector optimization.
\begin{definition}\cite{Luc}\label{D1}
It is said that a step size $\kappa>0$ satisfies
\begin{itemize}
\item  the  exact line search condition if
\begin{equation}\label{20-}
\lambda(u+\kappa d,d)=0;
\end{equation}
\item the standard Wolfe conditions if
\begin{equation*}\label{24-}\begin{aligned}
&\Phi(u+\kappa d)\preceq_E \Phi(u)+\rho\kappa\lambda(u,d)\xi,\\
&\lambda(u+\kappa d,d)\geq\sigma \lambda(u,d);
\end{aligned}\end{equation*}
\item  the strong Wolfe conditions if
\begin{equation}\label{3}\begin{aligned}
 &\Phi(u+\kappa d)\preceq_E \Phi(u)+\rho\kappa\lambda(u,d)\xi,\\
&|\lambda(u+\kappa d,d)|\leq\sigma |\lambda(u,d)|;
\end{aligned}\end{equation}
\end{itemize}
where $0<\rho<\sigma<1$ are given constants, $\xi\in E$ is such that
\begin{equation}\label{18-}
0<\langle \xi,v\rangle\leq 1,\qquad\forall v\in V.
\end{equation}
\end{definition}
We now introduce a new  vector extension of Wolfe-type conditions, based on the scalar form proposed in \cite{Che}.
\begin{definition}\label{D2}
We call that a step size $\kappa>0$ satisfies the generalized Wolfe conditions if
\begin{equation}\begin{aligned}\label{2*}
&\Phi(u+\kappa d)\preceq_E \Phi(u)+\rho\kappa\lambda(u,d)\xi,\\
&-\mu \lambda(u,d)\geq\lambda(u+\kappa d,d)\geq\sigma \lambda(u,d),
\end{aligned}\end{equation}
where $0<\rho<\sigma<1$ and $\mu\geq0$, and $\xi\in E$  satisfies (\ref{18-}).
\end{definition}

\begin{remark}
Usually, we choose  $\xi=\hat{\xi}/\hat{M}$ to satisfy (\ref{18-}), where $\hat{\xi}\in\text{int}(E)$ and $\hat{M}=\max\{\langle \hat{\xi},v\rangle:v\in V\}$. Particularly in the multiobjective optimization context,  taking $\xi=[1,1,\ldots,1]^T\in\mathbf{R}^m$ can satisfy (\ref{18-}).
\end{remark}

\section{Three-term PRP CG method for vector optimization}
In this part, we propose a novel three-term PRP CG method with sufficient descent property for vector optimization.

The vector extension of the PRP conjugate parameter is defined as
\begin{equation*}\label{22-} \beta_k^{PRP}:=\frac{-\lambda(x^{k},\vartheta(x^{k}))+\lambda(x^{k-1},\vartheta(x^{k}))}
{-\lambda(x^{k-1},\vartheta(x^{k-1}))},\end{equation*}
which was first proposed in \cite{Luc}  by Lucambio P\'{e}rez and  Prudente  for solving the unconstrained vector optimization problem (\ref{1}) with the recurrence
$$x^{k+1}=x^k+\alpha_kd^k\qquad\text{for}~k=0,1,2,\ldots.$$
In order for the vector extension of CG methods to be well-defined,   $d^k$ should be an $E$-descent direction of $\Phi$ at $x^k$, i.e., $\lambda(x^k,d^k)<0$. The more stringent condition, called the sufficient descent condition proposed in \cite{Luc}, is defined as
\begin{equation}\label{4}
\lambda(x^k,d^k)\leq h\lambda(x^k,\vartheta(x^k))
\end{equation}
for some $h\in(0,1]$.
In \cite{Luc}, under the descent assumption and the nonnegativity of conjugate parameters,  the authors establish  the convergence of the PRP CG method using the strong Wolfe  line search,
in which  the search direction is given as
\begin{equation}\label{6-}
d^k=\left\{\begin{array}{l}\vartheta(x^k)\qquad\text{for}\ k=0,\\
\vartheta(x^k)+\beta^{PRP+}_kd^{k-1}\qquad\text{for}\ k\ge1,
\end{array}\right.\end{equation}
where
$\beta_k^{PRP+}:=\max\{0,\beta_k^{PRP}\}$.

However, for the PRP CG method, the search direction  (\ref{6-})  may fail to satisfy (\ref{4}) and may not even be a descent  direction, as shown in the following case.
\begin{example}\label{e1}
Consider problem (\ref{1}), where $E= \mathbf{R}_+^{2}$, $V=\{[1,0]^T,[0,1]^T\}$ and $\Phi:\mathbf{R}^2\to \mathbf{R}^2$ is given by
$$ \Phi_1(x_1,x_2):=\frac{x_1^2+\sin x_2}{2}\quad\text{and}\quad
\Phi_2(x_1,x_2):=\frac{(x_1-1)^2-(x_2-1)^2}{2}.$$
Choose $x^0=(1.5,0.9)^T$. By solving problem (\ref{19-}), we obtain  $\vartheta(x^0)=(-0.5,-0.1)^T$. Hence, $d^0=\vartheta(x^0)=(-0.5,-0.1)^T$. It is clear that $\alpha_0=3.1669$ satisfies  (\ref{20-}), then $x^1=x^0+\alpha_0d^0=(-0.0835,0.5833)^T$. Therefore, $\vartheta(x^1)=(0.0835,-0.4173)^T$,
$$\beta_1^{PRP+}=\max\bigg\{0,\frac{-\lambda(x^{1},\vartheta(x^{1}))+\lambda(x^{0},\vartheta(x^{1}))}
{-\lambda(x^{0},\vartheta(x^{0}))}\bigg\}=0.6966,$$
and  $d^1=\vartheta(x^1)+\beta_1^{PRP+}d^0=(-0.2649,-0.4870)^T$. Since $\lambda(x^1,d^1)=0.0840>0$, we deduce that $d^1$ is not an $E$-descent direction of $\Phi$ at $x^1$.
\end{example}
The example shows that   the PRP CG method for vector optimization may fail to generate a descent  direction in some cases, even when using the exact line search. This leads to additional  search steps or  nonconvergence   for this vector extension. In \cite{Luc}, the authors provided a sufficient condition on conjugate parameters to guarantee the descent property of CG methods. However, an open problem in their work is how to ensure the descent property of  the PRP CG method. To deal with this problem, we propose a three-term PRP CG method for vector optimization without  modifying the conjugate parameters, where the search direction is defined as
\begin{equation}\label{9-}
d^k:=\left\{\begin{array}{l}\vartheta(x^k)\qquad\text{for}\ k=0,\\
\vartheta(x^k)+\beta^{PRP+}_kd^{k-1}
-\beta^{PRP+}_k\frac{|\lambda(x^k,d^{k-1})|}{\lambda(x^{k},\vartheta(x^{k}))}\vartheta(x^k)\qquad\textrm{for}\ k\ge1.
\end{array}\right.\end{equation}
The main conclusion in this part is as follows:
\begin{theorem}\label{T0}
If $x^k$ is not $E$-Pareto critical  and  the  search direction $d^k$ is taken as (\ref{9-}), then $d^k$ satisfies (\ref{4}) with $h=1$.
\end{theorem}
\begin{proof}
Since $x^k$ is not $E$-Pareto critical, we obtain $\vartheta(x^k)\neq0$ by Lemma \ref{L2}(ii). The condition (\ref{4}) trivially holds for $ d^0=\vartheta(x^0)$. For any $k\geq1$, it can be deduced from (\ref{9-}) that
\begin{equation}\label{7-}
\lambda(x^k,d^k)=\lambda\bigg(x^k,\vartheta(x^k)+\beta^{PRP+}_kd^{k-1}
-\beta^{PRP+}_k\frac{|\lambda(x^k,d^{k-1})|}{\lambda(x^{k},\vartheta(x^{k}))}\vartheta(x^k)\bigg).
\end{equation}
By Lemma \ref{L1}(i) and the fact $\lambda(x^k,\vartheta(x^k))<0$ obtained from Lemma \ref{L2}(ii), we can infer from (\ref{7-}) that
$$\begin{aligned}
\lambda(x^k,d^k)&\leq\lambda(x^{k},\vartheta(x^{k}))+\beta^{PRP+}_k\lambda(x^{k},d^{k-1})
-\beta^{PRP+}_k\frac{|\lambda(x^k,d^{k-1})|}{\lambda(x^{k},\vartheta(x^{k}))}\lambda(x^k,\vartheta(x^k))\\
&\leq\lambda(x^{k},\vartheta(x^{k}))+\beta^{PRP+}_k\lambda(x^{k},d^{k-1})-\beta^{PRP+}_k|\lambda(x^k,d^{k-1})|\\
&\leq\lambda(x^{k},\vartheta(x^{k})).
\end{aligned}$$
Combining these, we obtain $\lambda(x^k,d^k)\leq\lambda(x^{k},\vartheta(x^{k}))$, which establishes the theorem.
\end{proof}
\section{Global convergence}
In this section,  we present a three-term PRP CG algorithm  for vector optimization with a generalized Wolfe line  search  procedure, without self-adjusting strategies and regular restarts,  as follows:

\noindent\rule{\linewidth}{1.5pt}
{\bf Algorithm TT-PRP:} A three-term PRP CG method for vector optimization\\
\noindent\rule{\linewidth}{1.5pt}
\noindent{\bf Step 1.} Let $0<\rho<\sigma<1$, $\mu>0$, $\xi\in E$ as in (\ref{18-}). Compute $\vartheta(x^0)$ as in (\ref{2}). If $\vartheta(x^0)=0$, then stop. Otherwise, set $d^0:=\vartheta(x^0)$ and $k\leftarrow1$.

\noindent{\bf Step 2.} Compute a step size $\alpha_{k-1}>0$ such that
\begin{subequations}\label{1*}
\begin{equation}\label{5}
\Phi(x^{k})\preceq_E \Phi(x^{k-1})+\rho\alpha_{k-1}\lambda(x^{k-1},d^{k-1})\xi,
\end{equation}
\begin{equation}\label{6}
-\mu \lambda(x^{k-1},d^{k-1})\geq\lambda(x^k,d^{k-1})\geq\sigma \lambda(x^{k-1},d^{k-1}),
\end{equation}
\end{subequations}
where
\begin{equation}
x^{k}:=x^{k-1}+\alpha_{k-1}d^{k-1}.
\end{equation}

\noindent{\bf Step 3.} Compute $\vartheta(x^{k})$ as in (\ref{2}). If $\vartheta(x^{k})=0$, then  stop.

\noindent{\bf Step 4.}
Define
\begin{equation}\label{1-}
d^{k}:=\vartheta(x^{k})+\beta_{k}d^{k-1}-\beta_{k}\frac{|\lambda(x^{k},d^{k-1})|}{\lambda(x^{k},\vartheta(x^{k}))}
\vartheta(x^{k}),
\end{equation}
where
\begin{equation}\label{12-} \beta_{k}:=\max\bigg\{0,\frac{-\lambda(x^{k},\vartheta(x^{k}))+\lambda(x^{k-1},\vartheta(x^{k}))}
{-\lambda(x^{k-1},\vartheta(x^{k-1}))}\bigg\}.
\end{equation}
\noindent{\bf Step 5.} Let $k\leftarrow k+1$, and go to {\bf Step 2}.\\[-0.5mm]
\noindent\rule{\linewidth}{1.5pt}

We need the following general hypotheses:
\begin{itemize}
\item [$(H_1)$:]  The set $V$ is finite
\item [$(H_2)$:] There exists $M>0$ such that $\|x\|\leq M$ for all  $x\in\Delta:=\{x\in \mathbf{R}^n:\Phi(x)\preceq_E \Phi(x^0)\}$.
\item [$(H_3)$:] The Jacobian $J\Phi$ is $L$-Lipschitz continuous on an open set $\Omega\supset\Delta$
\end{itemize}

In what follows, we suppose that the convergence does not occur in  finite steps, i.e., $\vartheta(x^k)\neq0$ for all $k\ge0$. By Theorem \ref{T0}, we can easily check that  the search direction  generated by the Algorithm TT-PRP invariably satisfies (\ref{4}) with $h=1$.

Note that if $d^k$ is always an $E$-descent direction of $\Phi$ at $x^k$, then it can be deduced from assumption $(H_2)$ that $\{\lambda(x^k,\vartheta(x^k))\}$ is bounded. In fact, in this case, by the boundedness of $\{x^k\}$ and the continuity conclusions of Lemma \ref{L2}, there exist $\gamma,\theta>0$  such that
\begin{equation}\label{11}
\|J\Phi(x^k)\|\leq\gamma\quad\text{and}\quad\|\vartheta(x^k)\|\leq\theta,
\end{equation}
which, along with the definitions of $\lambda$ and $V$, yields that
\begin{equation}\label{25}
\begin{aligned}&|\lambda(x^k,\vartheta(x^k))|=|\sup\{\langle J\Phi(x^k)\vartheta(x^k),v\rangle:v\in V\}|\\
&\leq \sup\{\|J\Phi(x^k)\vartheta(x^k)\|\|v\|:v\in V\}\leq \|J\Phi(x^k)\|\|\vartheta(x^k)\|\leq \gamma\theta.\end{aligned}\end{equation}
We next present the  Zoutendijk  condition for vector optimization as follows:
\begin{lemma}\label{L3}
Under assumptions $(H_1)-(H_3)$, if the Algorithm TT-PRP generates an infinite sequence $\{x^k\}$, then
$$\sum_{k\geq0}\frac{\lambda^2(x^k,d^k)}{\|d^k\|^2}<\infty.$$
\end{lemma}
\begin{proof}
Clearly, $d^k\neq0$ and then $\lambda^2(x^k,d^k)/\|d^k\|^2$ is well-defined. Since $\alpha_k$ satisfies (\ref{2*}) for all $k\ge0$,   we can use \cite[Proposition 3.3]{Luc}  to obtain the desired conclusion.
\end{proof}

We will  mainly proceed to show that $\vartheta(x^k)$ cannot be bounded away from zero. To do this,  we need to make the following contradictory assumption:

$(H_4)$: There exists $\eta>0$ such that
\begin{equation*}
\|\vartheta(x^k)\|\geq\eta,\qquad \forall k\ge0.
\end{equation*}
Based on the contradictory assumption, we claim the following  auxiliary lemmas.

\begin{lemma}\label{L4}
Assume that assumptions $(H_1)-(H_4)$ hold, if the Algorithm TT-PRP generates an infinite sequence $\{x^k\}$, then
$$\sum_{k\ge0}\frac{1}{\|d^k\|^2}<\infty.$$
\end{lemma}
\begin{proof}
 By assumption $(H_4)$, Lemma \ref{L2}(ii) and Theorem \ref{T0}, we have
$$\frac{\eta^4}{4\|d^k\|^2}\leq\frac{\|\vartheta(x^k)\|^4}{4\|d^k\|^2}
\leq\frac{\lambda^2(x^k,\vartheta(x^k))}{\|d^k\|^2}\leq\frac{\lambda^2(x^k,d^k)}{\|d^k\|^2}$$
for all $k\geq0$, which, using Lemma \ref{L3}, yields that
$$ \frac{{\eta}^4}{4}\sum_{k\ge0}\frac{1}{\|d^k\|^2}<\infty.$$
\end{proof}

\begin{lemma}\label{L5}
Under assumptions $(H_1)-(H_4)$, if the Algorithm TT-PRP generates an infinite sequence $\{x^k\}$, then
$$\sum_{k\ge1}\|\varrho^k-\varrho^{k-1}\|^2<\infty,$$
where $\varrho^k:=d^k/\|d^k\|$.
\end{lemma}
\begin{proof}
Since (\ref{1-}) and the definition of $\lambda$, we have
$$\begin{aligned}
d^k=\|d^{k-1}\|\bigg(\frac{\vartheta(x^k)}{\|d^{k-1}\|}+\beta_k \varrho^{k-1}-\beta_k\frac{|\lambda(x^k,\varrho^{k-1})|}{\lambda(x^{k},\vartheta(x^{k}))}
\vartheta(x^k)\bigg).
\end{aligned}$$
From this, we can assume that
\begin{equation}\label{12}
\|d^k\|=\nu_k(x^k)\|d^{k-1}\|
\end{equation}
for all $k\geq1$, where
$$\nu_k(x^k):=\bigg\|\frac{\vartheta(x^k)}{\|d^{k-1}\|}+\beta_k \varrho^{k-1}-\beta_k\frac{|\lambda(x^k,\varrho^{k-1})|}{\lambda(x^{k},\vartheta(x^{k}))}
\vartheta(x^k)\bigg\|.$$
According to $\|\varrho^{k-1}\|=1$, the boundedness of $\{x^k\}$, the boundedness of $\{1/\|d^{k-1}\|\}$ obtained from Lemma \ref{L4}, and the continuity conclusions of Lemmas \ref{L1} and \ref{L2}, we deduce that $\{\nu_k(x^k)\}$ is bounded. By induction, it follows from (\ref{12}) that
$$\|d^k\|=\|d^{0}\|\prod^k_{i=1}\nu_i(x^i), $$
which leads to
\begin{equation}\label{G1}
\frac{1}{\|d^k\|^2}=\frac{1}{\|d^{0}\|^2\prod^k_{i=1}\nu^2_i(x^i)}.
\end{equation}
Combining  Lemma \ref{L4} with the boundedness of $\{\nu_k(x^k)\}$, we can infer from (\ref{G1}) that
there exists $\varepsilon>0$ such that
$$\nu_k(x^k)\geq\varepsilon,\qquad\forall k\ge1. $$
Form this, (\ref{12}) shows that
\begin{equation}\label{G2}
\|d^k\|\ge\varepsilon\|d^{k-1}\|,\qquad\forall k\ge1.
\end{equation}
Let
$$\varpi^k:=\frac{1}{\|d^{k}\|}\bigg(\vartheta(x^k)-\beta_k\frac{|\lambda(x^k,d^{k-1})|}{\lambda(x^{k},\vartheta(x^{k}))}
\vartheta(x^k)\bigg).$$
From (\ref{1-}), we have
\begin{equation}\label{13}
\varrho^k=\varpi^k+\beta_k\frac{\|d^{k-1}\|}{\|d^{k}\|}\varrho^{k-1}.
\end{equation}
Using the fact $\|\varrho^k\|=\|\varrho^{k-1}\|=1$, we deduce from (\ref{13}) that
$$\|\varpi^k\|=\bigg\|\varrho^k-\beta_k\frac{\|d^{k-1}\|}{\|d^{k}\|}\varrho^{k-1}\bigg\|
=\bigg\|\beta_k\frac{\|d^{k-1}\|}{\|d^{k}\|}\varrho^k-\varrho^{k-1}\bigg\|,$$
which, combined with $\beta_k\ge0$, gives that
\begin{equation}\label{14}
\begin{aligned}
\|\varrho^k-\varrho^{k-1}\|&\leq\bigg(\beta_k\frac{\|d^{k-1}\|}{\|d^{k}\|}+1\bigg)\|\varrho^k-\varrho^{k-1}\|\\
&=\bigg\|\varrho^k-\beta_k\frac{\|d^{k-1}\|}{\|d^k\|}\varrho^{k-1}+\beta_k\frac{\|d^{k-1}\|}{\|d^k\|}\varrho^k-\varrho^{k-1}\bigg\|\\
&\leq\bigg\|\beta_k\frac{\|d^{k-1}\|}{\|d^{k}\|}\varrho^k-\varrho^{k-1}\bigg\|
+\bigg\|\varrho^k-\beta_k\frac{\|d^{k-1}\|}{\|d^{k}\|}\varrho^{k-1}\bigg\|\\
&\leq2\|\varpi^k\|.
\end{aligned}\end{equation}
From Lemma \ref{L2}(ii) and  assumption $(H_4)$, we have
\begin{equation}\label{15-}
|\lambda(x^{k},\vartheta(x^{k}))|>\frac{\|\vartheta(x^k)\|^2}{2}\geq\frac{\eta^2}{2}.
\end{equation}
By (\ref{6}), (\ref{11}), (\ref{G2}) and (\ref{15-}), we obtain from (\ref{14}) the following estimate
$$\begin{aligned}
&\sum_{k\ge1}\|\varrho^{k}-\varrho^{k-1}\|^2\\
&\leq\sum_{k\ge1}\frac{4}{\|d^{k}\|^2}\bigg\|\vartheta(x^k)-\beta_k\frac{|\lambda(x^k,d^{k-1})|}
{\lambda(x^{k},\vartheta(x^{k}))}
\vartheta(x^k)\bigg\|^2\\
&\leq\sum_{k\ge1}\frac{8\|\vartheta(x^{k})\|^2}{\|d^k\|^2}
+\sum_{k\ge1}\frac{8\beta^2_k\lambda^2(x^{k},d^{k-1})\|\vartheta(x^{k})\|^2}{\lambda^2(x^{k},\vartheta(x^{k}))\|d^k\|^2}\\
&\leq\sum_{k\ge1}\frac{8\|\vartheta(x^{k})\|^2}{\|d^{k}\|^2}
+\sum_{k\ge1}\frac{8\beta^2_k\max\{\mu^2,\sigma^2\}\|\vartheta(x^{k})\|^2\lambda^2(x^{k-1},d^{k-1})}{\varepsilon^2\lambda^2(x^{k},\vartheta(x^{k}))\|d^{k-1}\|^2}\\
&\leq\sum_{k\ge1}\frac{8\theta^2}{\|d^{k}\|^2}
+\sum_{k\ge1}\frac{32\beta^2_k\max\{\mu^2,\sigma^2\}\theta^2\lambda^2(x^{k-1},d^{k-1})}{\varepsilon^2\eta^4\|d^{k-1}\|^2}\\
&<\infty.
\end{aligned}$$
Here, we use Lemmas \ref{L3} and \ref{L4}, and the proof is complete.
\end{proof}
We now proceed to establish the global convergence of the Algorithm TT-PRP.
\begin{theorem}\label{T2}
Suppose that assumptions $(H_1)-(H_3)$ hold, if the Algorithm TT-PRP generates an infinite sequence $\{x^k\}$, then
$$ \liminf_{k\to\infty}\|\vartheta(x^k)\|=0.$$
\end{theorem}
\begin{proof}
Suppose, contrary to our claim, that assumption $(H_4)$ holds. By  Lemma \ref{L1}(iii), (\ref{11}) and (\ref{15-}), it follows from (\ref{12-}) that
\begin{equation}\label{15}\begin{aligned}
|\beta_{k}|&=\frac{|-\lambda(x^{k},\vartheta(x^{k}))+\lambda(x^{k-1},\vartheta(x^{k}))|}
{|\lambda(x^{k-1},\vartheta(x^{k-1}))|}\\
&\leq\frac{L\|\vartheta(x^k)\|\|x^{k}-x^{k-1}\|}
{|\lambda(x^{k-1},\vartheta(x^{k-1}))|}\\
&\leq\frac{2L\theta}
{{\eta}^2}\|x^{k}-x^{k-1}\|.
\end{aligned}\end{equation}
Moreover, it can be inferred from  (\ref{25}) and (\ref{15-})  that
\begin{equation}\label{16-}
|\beta_{k}|\leq\frac{|\lambda(x^{k},\vartheta(x^{k}))|+|\lambda(x^{k-1},\vartheta(x^{k}))|}
{|\lambda(x^{k-1},\vartheta(x^{k-1}))|}\leq\frac{4\gamma\theta}{\eta^2}.
\end{equation}
If $ \lambda(x^k,d^{k-1})\ge0$,  we  have
\begin{equation*}\label{13-}|\lambda(x^k,d^{k-1})|\leq |\lambda(x^{k},d^{k-1})-\lambda(x^{k-1},d^{k-1})|.
\end{equation*}
If $ \lambda(x^k,d^{k-1})\leq0$,  we deduce from (\ref{6}) that
$$|\lambda(x^k,d^{k-1})|\leq\sigma|\lambda(x^{k-1},d^{k-1})|,$$
which leads to
\begin{equation*}\label{14-}\begin{aligned}|\lambda(x^k,d^{k-1})|
&\leq\frac{\sigma}{1-\sigma}(|\lambda(x^{k-1},d^{k-1})|-|\lambda(x^{k},d^{k-1})|)\\
&\leq\frac{\sigma}{1-\sigma}|\lambda(x^{k},d^{k-1})-\lambda(x^{k-1},d^{k-1})|.
\end{aligned}\end{equation*}
Combining these and using Lemma \ref{L1}(iii), we see that
\begin{equation}\label{16}\begin{aligned}|\lambda(x^k,d^{k-1})|
&\leq C_1|\lambda(x^{k},d^{k-1})-\lambda(x^{k-1},d^{k-1})|\\
&\leq C_1L\|x^{k}-x^{k-1}\|\|d^{k-1}\|,\end{aligned}\end{equation}
where $C_1:\max\{1,\sigma/(1-\sigma)\}$.  Combining (\ref{11}) with (\ref{15-})-(\ref{16}), we infer from (\ref{1-}) that
\begin{equation}\label{4-}\begin{aligned}
\|d^k\|^2&\leq\bigg(\|\vartheta(x^k)\|+\beta_k\|d^{k-1}\|+
\frac{\beta_k|\lambda(x^k,d^{k-1})|\|\vartheta(x^{k})\|
}{|\lambda(x^{k},\vartheta(x^{k}))|}\bigg)^2\\
&\leq\bigg(\theta+\frac{2L\theta}
{{\eta}^2}\|x^{k}-x^{k-1}\|\|d^{k-1}\|+\frac{8C_1L\gamma\theta^2}{\eta^4}\|x^{k}-x^{k-1}\|\|d^{k-1}\|\bigg)^2\\
&\leq\bigg(\theta+\frac{2L\eta^2+8C_1L\gamma\theta^2}{\eta^4}\|x^{k}-x^{k-1}\|\|d^{k-1}\|\bigg)^2\\
&\leq2{\theta}^2+2{C}^2
\|x^{k}-x^{k-1}\|^2
\|d^{k-1}\|^2
\end{aligned}\end{equation}
for all $k\ge1$, where $C:=(2L\theta{\eta}^2+8C_1L\gamma\theta^2)/{\eta}^4$.

For any positive integer $N\ge\lceil8CM\rceil$,  where $\lceil\cdot\rceil$ is the  ceiling function, by Lemma \ref{L5}, there exists a positive integer $m_0$ such that
\begin{equation}\label{19}
\sum_{i\geq m_0+1}\|\varrho^i-\varrho^{i-1}\|^2\leq\frac{1}{4N}.
\end{equation}
By the definition of $\varrho^k$, we deduce that
\begin{equation*}\label{17}\begin{aligned}
x^m-x^{m_0}&=\sum^m_{i=m_0+1}(x^{i}-x^{i-1})=\sum^m_{i=m_0+1}\|s^{i-1}\|\varrho^{i-1}\\
&=\sum^m_{i=m_0+1}\|s^{i-1}\|\varrho^{m_0}+\sum^m_{i=m_0+1}\|s^{i-1}\|(\varrho^{i-1}-\varrho^{m_0})
\end{aligned}\end{equation*}
for all $m\ge m_0+1$, where $s^{i-1}:=x^i-x^{i-1}$. From this, we get
\begin{equation}\label{18}\begin{aligned}
\sum^m_{i=m_0+1}\|s^{i-1}\|&\leq\|x^m-x^{m_0}\|+\sum^m_{i=m_0+1}\|s^{i-1}\|\|\varrho^{i-1}-\varrho^{m_0}\|\\
&\leq2M+\sum^m_{i=m_0+1}\|s^{i-1}\|\|\varrho^{i-1}-\varrho^{m_0}\|.
\end{aligned}\end{equation}
Since the arbitrariness of $N$  and the fact $(\sum^n_{i=1}a_i)^2\leq n\sum^n_{i=1}a_i^2 $ with $a_i\in\mathbf{R}$, it can be inferred from (\ref{19}) that
$$\begin{aligned}\|\varrho^m-\varrho^{m_0}\|\leq\sum^{m}_{i=m_0+1}\|\varrho^i-\varrho^{i-1}\|\leq\sqrt{m-m_0}\bigg(\sum^{m}_{i=m_0+1}\|\varrho^i-\varrho^{i-1}\|^2\bigg)^{\frac{1}{2}}
\leq\sqrt{N}\bigg(\frac{1}{4N}\bigg)^{\frac{1}{2}}=\frac{1}{2}\end{aligned}$$
for all $m\ge m_0+1$ which, combined with (\ref{18}), means that
\begin{equation}\label{20}
\sum^{m}_{i=m_0+1}\|s^{i-1}\|\leq4M.
\end{equation}
It follows from (\ref{4-}) by induction that
\begin{equation}\label{21}
\|d^m\|^2\leq2{\theta}^2\bigg(\sum^{m+1}_{i=m_0+1}\prod^{m}_{j=i}2C^2\|s^{j-1}\|^2\bigg)
+\|d^{m_0}\|^2\prod^{m}_{j=m_0+1}2{C}^2\|s^{j-1}\|^2
\end{equation}
for all $m\ge m_0+1$. Here, $ \prod=1$ when the index range is empty. Using the  mean inequality and (\ref{20}), we see that
\begin{equation}\label{22}\begin{aligned} &\prod^{m_0+p}_{j=m_0+1}2{C}^2\|s^{j-1}\|^2=\bigg(\prod^{m_0+p}_{j=m_0+1}\sqrt{2}C\|s^{j-1}\|\bigg)^2\\
&\leq\bigg(\frac{\sum^{m_0+p}_{j=m_0+1}\sqrt{2}C\|s^{j-1}\|}{p}\bigg)^{2p}
\leq\bigg(\frac{4\sqrt{2}CM}{p}\bigg)^{2p}\leq\frac{1}{2^{p}}
\end{aligned}\end{equation}
for all $p\geq\lceil8CM\rceil$. By (\ref{22}), it follows from (\ref{21}) that
\begin{equation*}\label{23}\begin{aligned}
\|d^m\|^2&\leq2{\theta}^2\bigg(\sum^{m_0+\lceil8CM\rceil}_{i=m_0+1}
\prod^{m_0+\lceil8CM\rceil-1}_{j=i}2C^2\|s^{j-1}\|^2\bigg)\\
&\quad+2\theta^2\bigg(\sum^{m+1}_{i=m_0+\lceil8CM\rceil+1}\prod^{m}_{j=i-\lceil8CM\rceil}2C^2\|s^{j-1}\|^2\bigg)
+\|d^{m_0}\|^2\prod^{m}_{j=m_0+1}2C^2\|s^{j-1}\|^2\\
&\leq2\theta^2\bigg(\sum^{m_0+\lceil8CM\rceil}_{i=m_0+1}
\prod^{m_0+\lceil8CM\rceil-1}_{j=i}2C^2\|s^{j-1}\|^2\bigg)\\
&\quad+2\theta^2\bigg(\sum^{m+1}_{i=m_0+\lceil8CM\rceil+1}\frac{1}{2^{m+\lceil8CM\rceil+1-i}}\bigg)+\frac{\|d^{m_0}\|^2}{2^{m-m_0}}\\
&\leq2\theta^2\bigg(\sum^{m_0+\lceil8CM\rceil}_{i=m_0+1}
\prod^{m_0+\lceil8CM\rceil-1}_{j=i}2C^2\|s^{j-1}\|^2\bigg)+ 4\theta^2+\|d^{m_0}\|^2
\end{aligned}\end{equation*}
for  all $m\geq m_0+\lceil8CM\rceil+1$, which implies that $d^m$ is bounded independently of $m$. This contradicts Lemma \ref{L5}, and  the theorem is proved.\end{proof}
\section{Numerical experiments}
In this section,  we give some numerical results in the multiobjective optimization context to show the practical behavior of the Algorithm TT-PRP. Hence, in what follows, let $E:=\mathbf{R}_+^{m}$, $V$ be the canonical basis of $\mathbf{R}^m$ and $\xi:=[1,1,\ldots,1]^T\in\mathbf{R}^m$. All codes are written in MATLAB R2023b and run on a PC with  CPU 2.50 GHz, 16 GB RAM.   The considered methods  are as follows:
\begin{itemize}
\item
TT-PRP: The conjugate parameter and the step size are given by the Algorithm TT-PRP.

\item
TT-PRP1: The conjugate parameter is chosen as (\ref{1-}), and the step size satisfies (\ref{3}).
\item
PRP+ \cite{Luc}: The conjugate parameter is defined as $\max\{\beta^{PRP}_k,0\}$, and the step size satisfies (\ref{3}).
\item  SD \cite{Fli2}: The steepest descent method with a strong Wolfe line search procedure.
\end{itemize}

In our setting, the parameters in line search procedures are chosen as $\rho=0.0001$, $\sigma=0.1$ and $\mu=0.2$. The  implementation of line search processes  is detailed in \cite{Luc2}.
The  maximum allowed  number  of iterations  is set to 3000.
By Lemma \ref{L2}, we deduce that $\vartheta(x^k)=0$ iff $\Theta(x^k)=0$ and $\vartheta(x^k)\neq0$ iff $\Theta(x^k)<0$. From these, we can stop the algorithm at $x^k$ declaring convergence whenever $$\Theta(x^k)\geq -5\times\text{eps}^{\frac{1}{2}},$$
where $\Theta(x^k)=\lambda(x^k,\vartheta(x^k))+\|\vartheta(x^k)\|^2/2$ and $\text{esp}\approx2.22\times 10^{-16}$. $\vartheta(x^k)$ and  $\Theta(x^k)$ can be obtained by using the  quadprog function  in Matlab to solve  problem (\ref{19-}).

All considered problems are listed in Table 1. The columns ``Problem'' and ``Source'' represent the problem name and the corresponding reference, respectively.   The columns ``$n$", ``$m$'' and ``Convex'' represent the number of variables,   the number of objectives and the convexity  of the  test problems, respectively.
The initial points are taken from a box given in the column  ``$x_D$''.

\begin{table}[ht]
\centering
\caption{Test problems}
\vspace{5pt}
\begin{tabular}{l c c c c c c}
\toprule
Problem~~~&~~~Source~~~&~~~$n$~~~&~~~$m$~~~&~~~Convex~~~&~~~$x_D$\\
\midrule
AP3  & \cite{Ans} & $2$ & $2$ & N & $[-2,2]^n$            \\[1mm]
Far1   & \cite{Hub} & $2$     & $2$       & N & $[-1,1]^n$  \\[1mm]
FDS-1  & \cite{Fli} & $2$ & $3$ & Y & $[-2,2]^n$    \\[1mm]
FDS-2  & \cite{Fli}& $100$ & $3$ & Y &  $[-2,2]^n$\\[1mm]
FDS-3  & \cite{Fli}& $150$ & $3$ & Y &  $[-2,2]^n$\\[1mm]
Hil1 & \cite{Hil} & $2$ & $2$ & N & $[0,1]^n$    \\[1mm]
Lov3   & \cite{Lov} & $2$       & $2$       & N & $[-100,100]^n$  \\[1mm]
Lov4  & \cite{Lov} & $2$ & $2$ & N & $[-100,100]^n$   \\[1mm]
MGH16-1 & \cite{Mor} & $4$ & $50$ & N & $[-25,25]\bigcup[-5,5]\bigcup[-5,5]\bigcup[-1,1]$      \\[1mm]
MGH16-2 & \cite{Mor} & $4$ & $100$ & N & $[-25,25]\bigcup[-5,5]\bigcup[-5,5]\bigcup[-1,1]$      \\[1mm]
MGH26 & \cite{Mor} & $4$ & $4$ & N & $[-1,1]^n$      \\[1mm]
MMR5-1 & \cite{Mig} & 1000 & 2 & N & $[-10,10]^n$    \\[1mm]
MMR5-2 & \cite{Mig} & 200 & 2 & N & $[-100,100]^n$    \\[1mm]
MOP5   & \cite{Hub} & $2$       & $3$       & N & $[-1,1]^n$  \\[1mm]
MOP7  & \cite{Hub} & $2$ & $3$ & Y & $[-400,400]^n$   \\[1mm]
SLC2-1   & \cite{Sch} &   $1000$     & $2$       & Y & $[-10,10]^n$  \\[1mm]
SLC2-2   & \cite{Sch} &   $200$     & $2$       & Y & $[-100,100]^n$ \\[1mm]
SLC2-3   & \cite{Sch} &   $1000$     & $2$       & Y & $[-100,100]^n$ \\
\bottomrule
\end{tabular}
\end{table}

 By solving  each problem  100 times with starting points randomly generated in the specified box,  we provide the performance in Table 2 under the metrics of the success rate (\%), the median number of iterations (mit),  the median number  of function evaluations (mf) and the median number of gradient evaluations (mg).

\begin{table}
\footnotesize
\centering
\caption{Performance of the considered  methods with the indicators \%, mit, mf and mg.}
\vspace{5pt}
\begin{tabular}{|ccccc|ccccc|}
\cline{1-10}
   & TT-PRP & TT-PRP1 & PRP+ & SD  &  & TT-PRP & TT-PRP1 & PRP+ & SD   \\
\hline
                   AP3 &      &        &        &         &MGH16-2 &      &      &      &    \\
\hline
               \%     & 100.0  & 100.0  & 100.0  & 100.0  & \%     &100.0 &100.0 &98.0 &86.0\\
               eit     & 7.0 & 7.0  & 7.0  & 69.0           &eit     &39.0   &42.5   &58.0   &176.0     \\
                mf  & 46.0& 51.0 & 55.0 & 468.5       & mf  &306.0  &358.5  &591.0  &3392.0          \\
                mg  & 37.0& 42.0 & 47.0 & 395.0 &mg  &267.5  &317.0  &526.5  &3182.5 \\
                \hline
                Far1 &   &&  &                            &MGH26& & & & \\
\hline

            \%     & 100.0  & 100.0  & 100.0  & 100.0 & \%     & 100.0  & 100.0  & 100.0  & 100.0  \\
               eit     & 33.0 & 32.5 & 38.0 & 70.5   & eit & 6.0 & 6.0 &6.0& 8.5 \\
               mf  & 276.0& 308.0& 383.5& 523.5 & mf & 27.0 & 31.0 &37.0&64.0\\
  mg  & 242.0& 273.0& 339.0 & 451.0 & mg & 19.5 & 24.0&29.0& 53.0\\
  \hline
FDS-1 &&&&                                             &MMR5-1 & & & &\\
\hline
          \%     &100.0 &100.0 &100.0 &100.0  & \%     & 100.0  & 100.0  & 100.0  & 100.0 \\
            eit     &6.0   &6.0   &6.0   &8.0   & eit & 91.5 & 97.0 &185.0& 217.0  \\
            mf  &51.0  &59.0  &61.0  &84.0 & mf & 525.0 & 655.0 &1305.0&1493.0\\
               mg  &43.0  &51.0  &53.0  &73.0 & mg & 426.5 & 543.5&1122.5& 1276.5 \\

  \hline
 FDS-2 & &&&                                  &MMR5-2& &  && \\
\hline
              \%     & 100.0 & 100.0 & 100.0 & 100.0 & \%     & 100.0  & 100.0  & 100.0  & 100.0  \\
             eit     & 123.5 & 138.0&171.5 & 221.5     & eit & 131.5 & 141.5 &223.0& 285.0 \\
               mf  & 879.0&1291.0&1537.0&2687.5 & mf & 703.0 & 840.5 &1712.5&2304.5\\
             mg  &745.0&1168.5&1335.0&2464.0& mg & 574.5 & 697.0&1492.0& 1985.0\\
 \hline
    FDS-3 &   &&&                                      &MOP5 & &  &&      \\
\hline
          \%     &100.0 &100.0 &100.0 &100.0  & \%     & 100.0  & 100.0  & 100.0  & 100.0\\
            eit     &126.0   &144.5   &179.0   &242.5   & eit & 2.0 & 2.0 &3.0& 3.0  \\
               mf  &890.5  &1334.5  &1661.5  &2875.5  & mf & 19.0 & 21.0 &28.0&28.0  \\
              mg  &756.0  &1178.0  &1440.0  &2653.5 & mg & 15.0 & 17.0&23.0& 23.0\\

  \hline
Hil1 &   &&&                                    &MOP7 & & & &                         \\
\hline
          \%     &100.0 &100.0 &100.0 &100.0  & \%     & 100.0  & 100.0  & 100.0  & 100.0  \\
            eit     &6.5   &7.0   &10.0   &16.5  & eit & 7.0 & 7.0 &8.0& 28.5   \\
               mf  &38.0  &56.0  &78.0  &104.0  & mf & 36.5 & 48.0 &58.0&169.0 \\
              mg  &29.5  &47.0  &67.0  &85.0 & mg & 27.5 & 39.0&48.0& 138.0\\
\hline
Lov3 &   &&&                                      &SLC2-1 & &  &&      \\
\hline
          \%     &100.0 &100.0 &100.0 &100.0  & \%     & 100.0  & 100.0  & 100.0  & 100.0\\
            eit     &2.0   &2.0   &3.0   &3.0   & eit & 18.0 & 18.5 &34.0& 63.5  \\
               mf  &18.0  &18.0  &26.0  &26.0  & mf & 238.0 & 303.5 &538.0&933.0  \\
              mg  &14.0  &14.0  &21.0  &21.0 & mg & 218.0 & 283.0&498.5& 872.5\\
\hline
Lov4 &  & &&                                        &SLC2-2& & &&\\
\hline
          \%     &100.0 &100.0 &100.0 &100.0 & \%     & 100.0  & 100.0  & 100.0  & 100.0\\
            eit     &1.0   &1.0   &1.0   &1.0  & eit & 18.5 & 20.0 &33.0& 61.0    \\
               mf  &6.0  &6.0  &6.0  &6.0  & mf & 230.0 & 336.0 &360.5&813.0  \\
              mg  &5.0  &5.0  &5.0  &5.0 & mg & 209.5 & 314.0&325.5& 739.5\\
\hline
MGH16-1 &   &&&                                      &SLC2-3 & &  &&      \\
\hline
          \%     &100.0 &100.0 &98.0 &93.0  & \%     & 100.0  & 100.0  & 100.0  & 100.0\\
            eit     &34.0   &36.0   &48.5   &97.0   & eit & 24.0 & 26.5 &42.0& 76.0  \\
               mf  &265.0  &293.0  &479.0  &1762.5  & mf & 338.0 & 362 &558.0&879.5  \\
              mg  &231.5  &252.0  &421.0  &1655.0 & mg & 309.5 & 331.5&495.5& 780.5\\
\hline
\end{tabular}
\end{table}

As observed, the TT-PRP and TT-PRP method exhibits well convergence  for all test problems. For  problems AP3, FDS-1, Lov3, Lov4, MGH26, MOP5 and MOP7, the   TT-PRP and TT-PRP1 methods show  comparable performance to the PRP+ method. In some small-scale nonconvex problems, such as Far1 and  Hil1, the  TT-PRP and TT-PRP1 methods exhibit a slight advantage over  the PRP+ method. For large-scale  problems, such as FDS-2, FDS-3, MGH16-1, MGH16-2,  MMR5-1, MMR5-2, SLC2-1, SLC2-2 and SLC2-3, the methods TT-PRP  and TT-PRP1 outperform the  PRP+ and SD methods. This suggests that the descent property and the conjugate property are  exhibited by the  TT-PRP and TT-PRP1 methods, which are superior to the  PRP+ and SD methods.
Compared with the TT-PRP1 method, the  TT-PRP method has a smaller  median number of iterations  in most test problems, but the median numbers of function evaluations and gradient evaluations seem to be lower when the median number of iterations is similar. This indicates that,  due to the relaxation of the line search,  the TT-PRP method can achieve the  appropriate step size  with  fewer  gradient evaluations and function evaluations than the TT-PRP1 method.

To  show a clearer  comparison of the performance of the considered methods, we use the performance profile proposed in \cite{Dol}. Let  $S$ be the method set and $P$ be the problem set. Denote $t_{p,s}$ as the performance measurement of method $s\in S$ on problem $p\in P$. Define $\rho_s:[1,\infty)\to[0,1]$ by
$$\rho_s(\omega):=\frac{|\{p\in P: r_{p,s}\le\omega\}|}{|P|}, $$
where $ r_{p,s}:=t_{p,s}/\min\{t_{p,s}:s\in S\}$, called the performance ratio. Clearly,  $\rho_s(1)$ represents the proportion of problems in the set $P$ that are solved most efficiently by method $s$,  and a larger value of $\rho_s(\omega)$ indicates the better performance of method $s$. The  comparison of  the considered methods in terms of efficiency and robustness  can be observed  on the left and right vertical axes of the matched performance profiles, respectively. The performance profiles are presented in Figure 1 using 100  random initial points for all test problems with the performance measurements: (a) CPU time; (b) Count of iterations; (c) Count of gradient evaluations; (d) Count of function evaluations.

Figure 1  shows that the TT-PRP and TT-PRP1 methods obviously exhibit higher  efficiency and robustness  than the PRP+ and SD methods for all considered measurement terms. In addition, under comparable performance curves in terms of  the count of iterations, the TT-PRP method  demonstrates  better performance over the  TT-PRP1  method with respect to the number of gradient  and  function evaluations. This  suggests again that the generalized Wolfe line search reduces the computational cost of  gradient and function evaluations compared to the strong Wolfe line search due to the relaxation.

Overall, the TT-PRP method exhibits superior performance compared to the other three methods, as shown in Table 2 and Figure 1. Therefore, the proposed method shows great  potential for practical applications.

\begin{figure}
    \centering
    \setcounter{subfigure}{0}
    \subfigure[CPU time]{
        \includegraphics[width=0.47\textwidth]{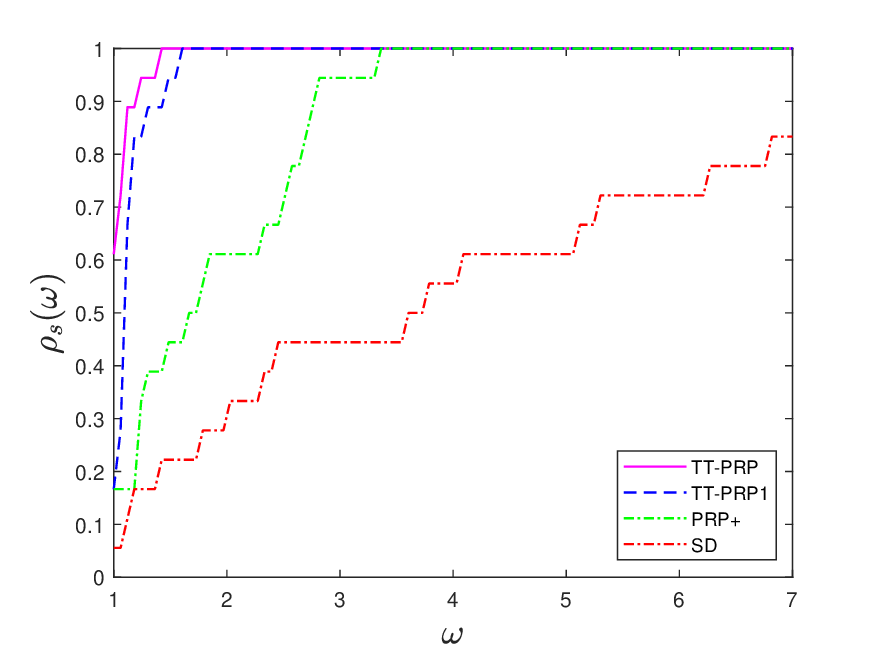}
    }
    \subfigure[Count of iterations] {
        \includegraphics[width=0.47\textwidth]{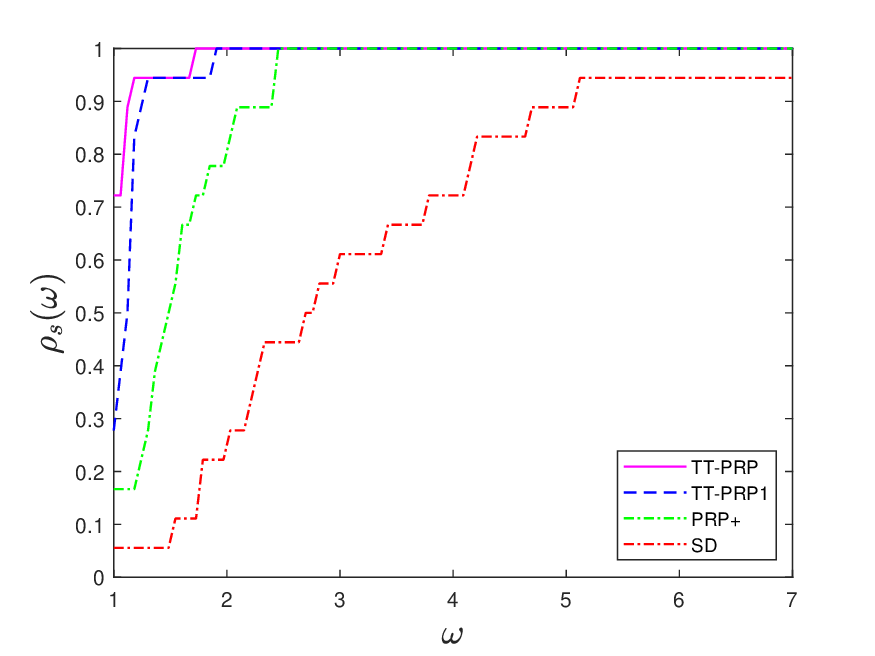}
    }\\[1mm]
    \subfigure[Count of gradient evaluations]{
        \includegraphics[width=0.47\textwidth]{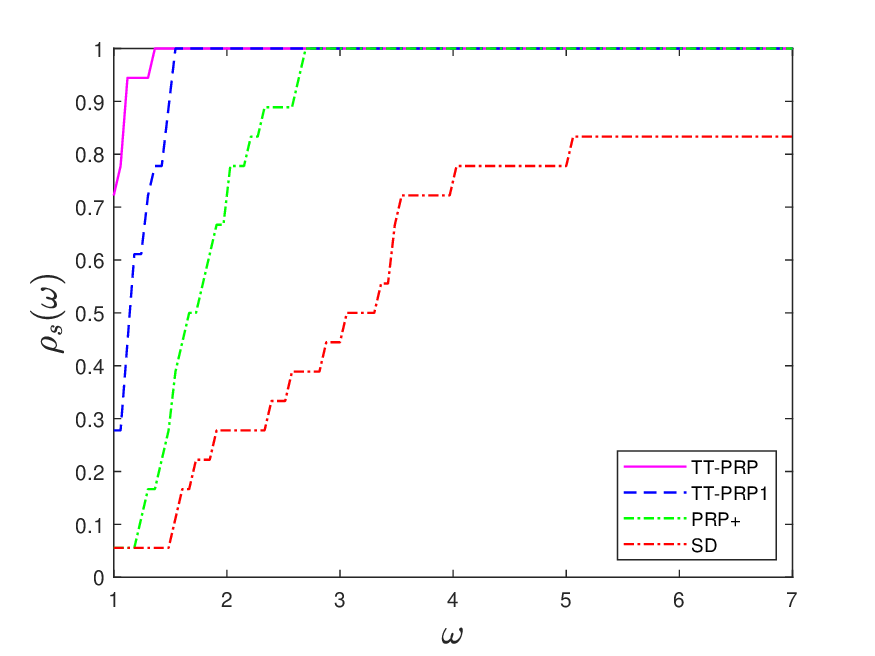}
    }
    \subfigure[Count of function evaluations]{
        \includegraphics[width=0.47\textwidth]{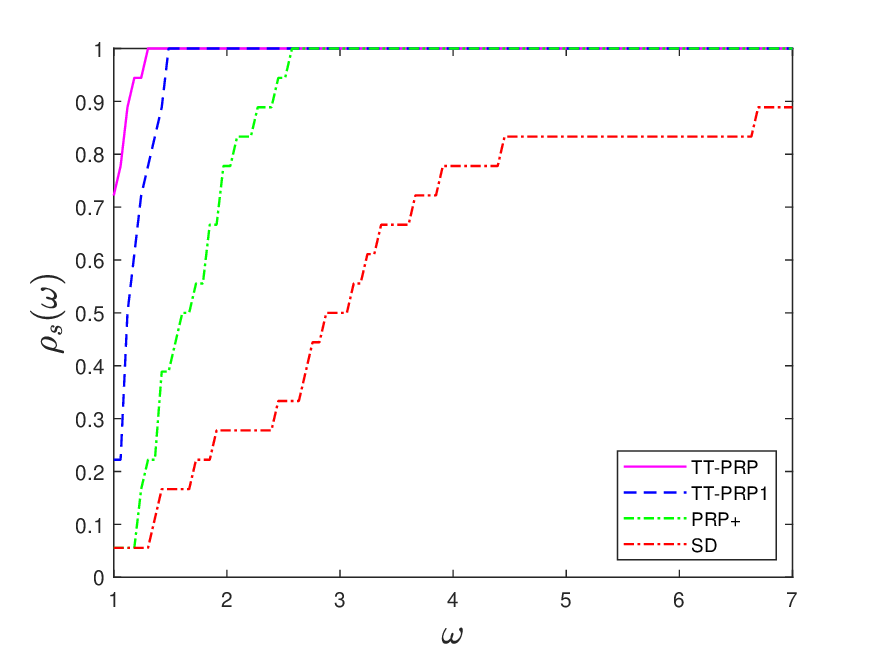}
    }
    \caption{Performance profiles based on the performance measurements (a)-(d) }
\end{figure}

The ability  to successfully approximate Pareto frontiers is an important evaluation criterion for vector optimization algorithms. To evaluate the performance of the TT-PRP method in generating Pareto frontiers, the value spaces of the bicriteria problems (AP3, Far1 and Hil1) and the three-criteria problems (FDS-1, MOP5 and MOP7) are plotted in Figures 2 and 3  using the TT-PRP method.  For these problems, 400 starting points are randomly generated (except for Far1, which uses 600  starting points due to the complexity of the image space) in the given box to plot the Pareto frontiers.  In these figures, a line segment represents a complete iteration, where the green dot denotes the final iteration point and the opposite end of the line segment represents the corresponding initial point.
Figures 2 and 3 show that
the TT-PRP method can give an  effective estimation of Pareto frontiers for the considered problems  using a reasonable number of starting points.
\begin{figure}
    \centering
    \setcounter{subfigure}{0}
    \subfigure[Image set of AP3]{
        \includegraphics[width=0.45\textwidth]{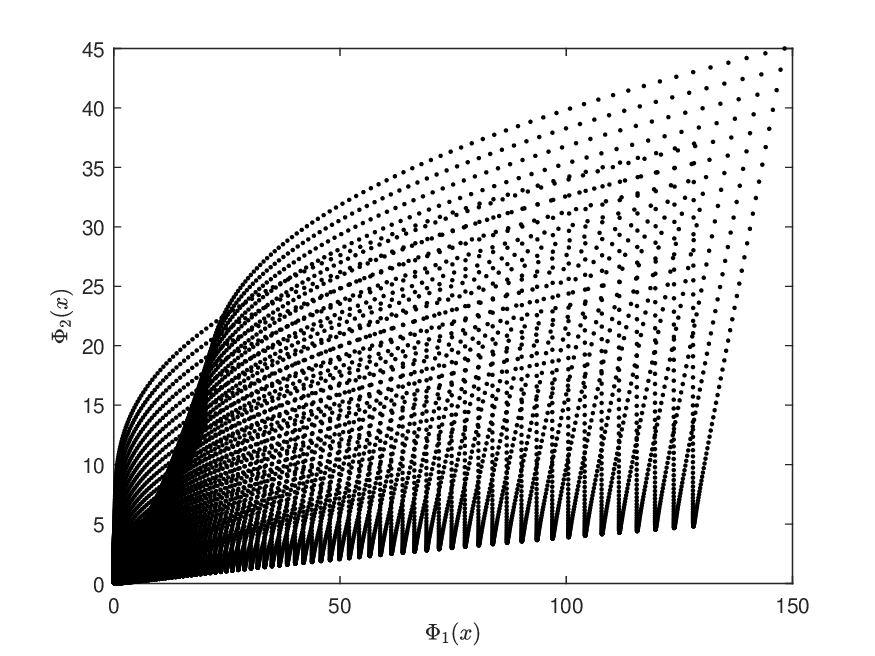}
    }
    \subfigure[AP3] {
        \includegraphics[width=0.45\textwidth]{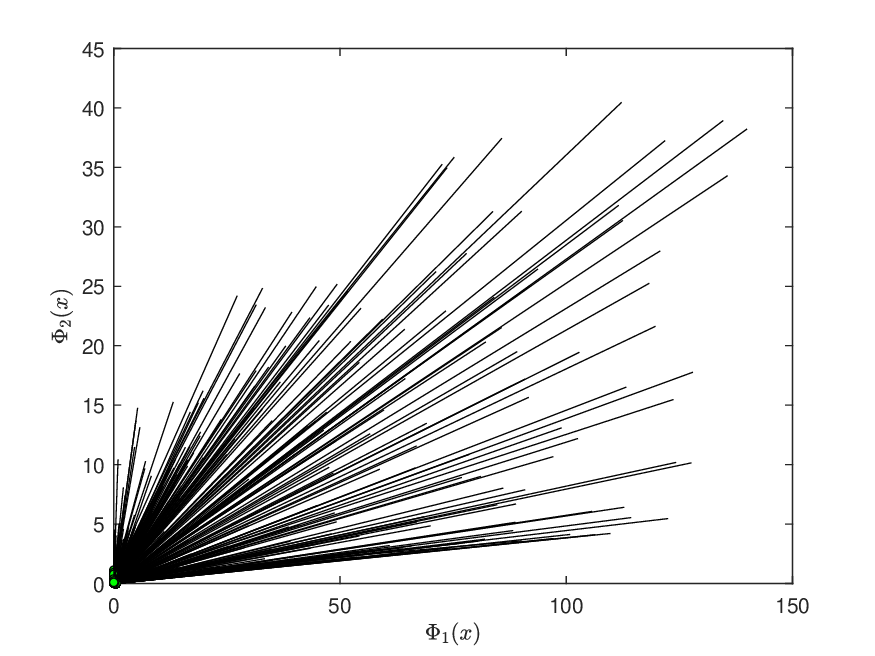}
    }\\[1mm]

    \subfigure[Image set of Far1]{
        \includegraphics[width=0.45\textwidth]{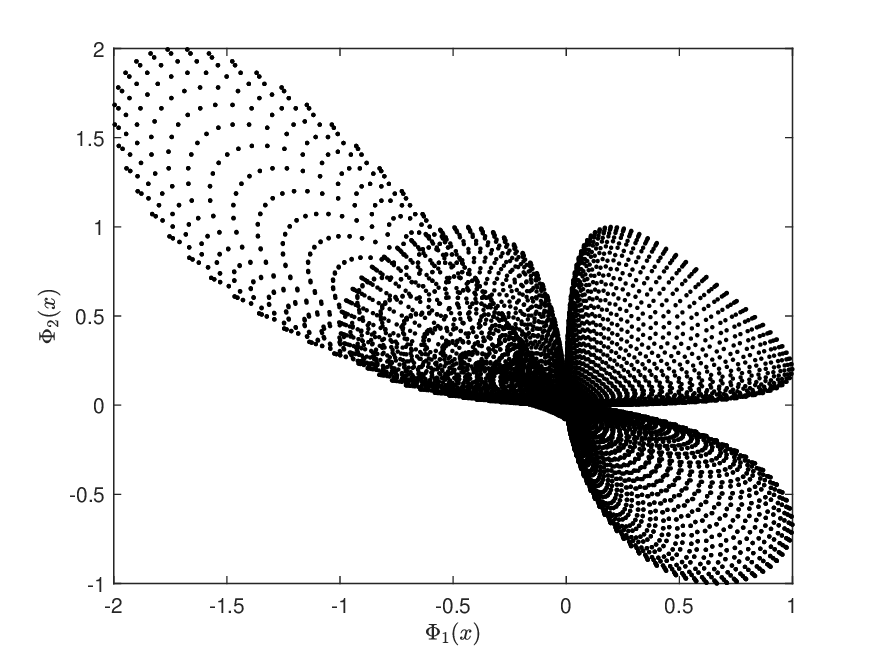}
    }
    \subfigure[Far1]{
        \includegraphics[width=0.45\textwidth]{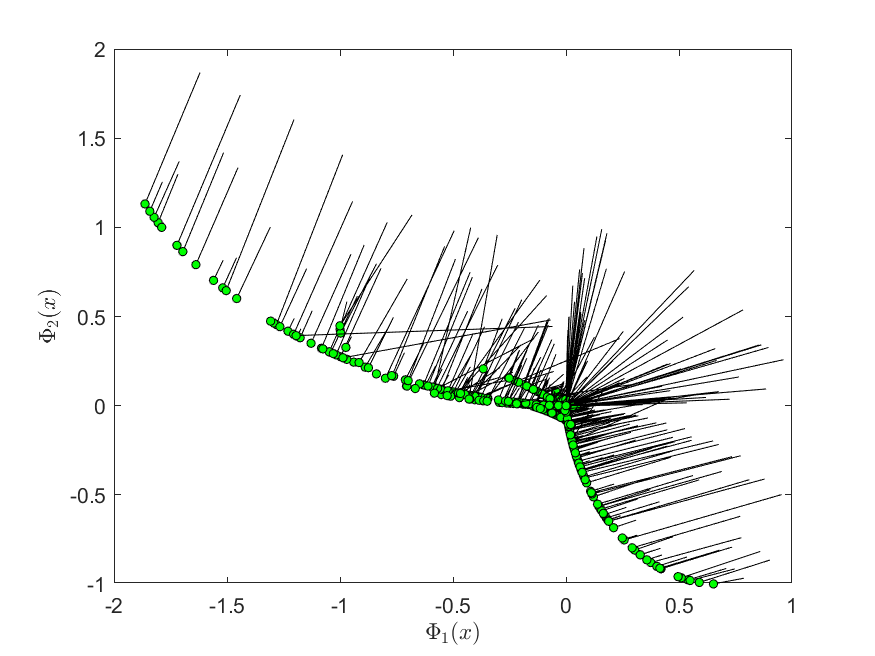}
    }\\[1mm]

    \subfigure[Image set of Hil1]{
        \includegraphics[width=0.45\textwidth]{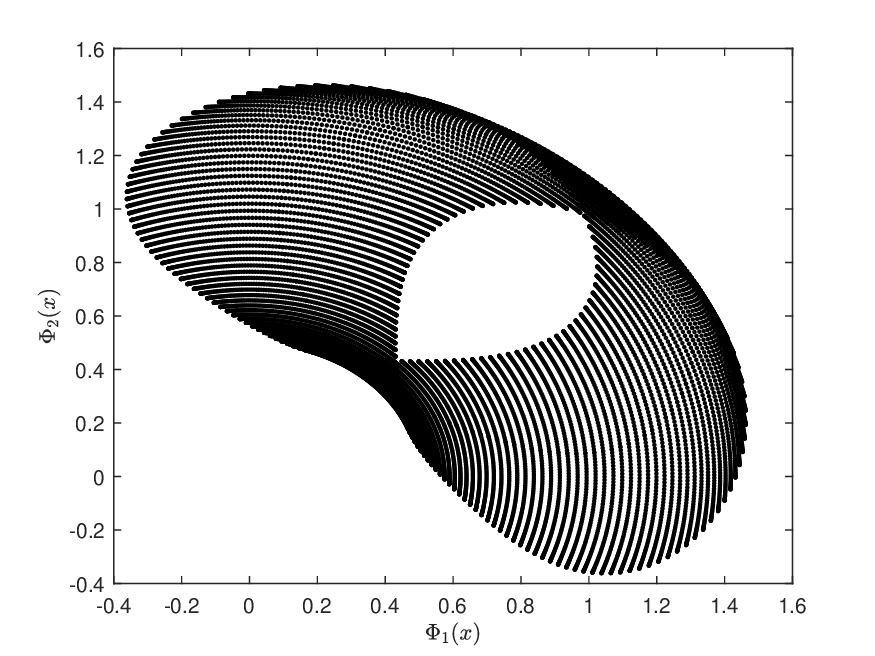}
    }
    \subfigure[Hil1]{
        \includegraphics[width=0.45\textwidth]{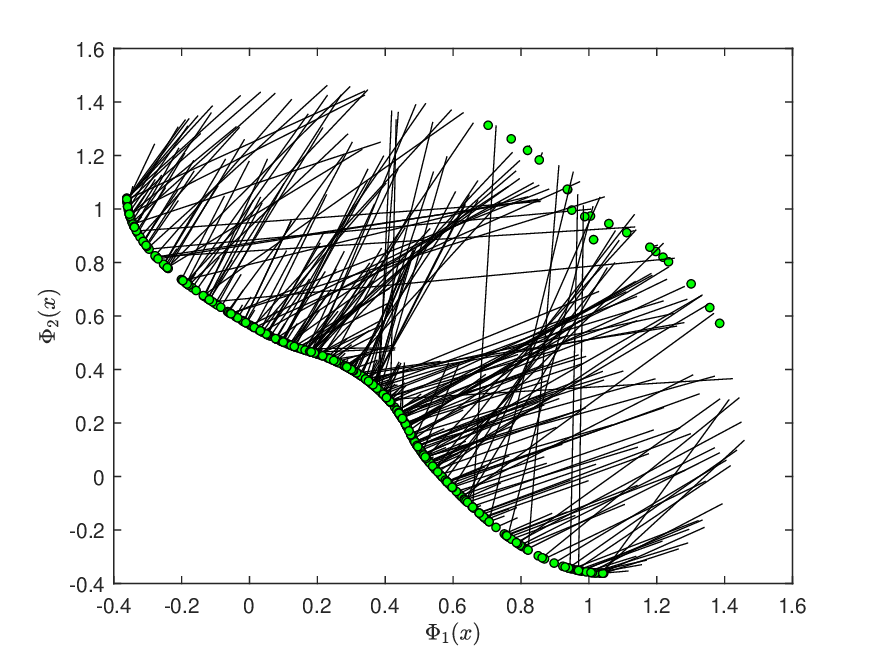}
    }\caption{ Approximations of  Pareto frontiers generated by the TT-PRP   method and  corresponding  image sets   for bicriteria problems. }
\end{figure}
\begin{figure}
    \centering
    \setcounter{subfigure}{0}
    \subfigure[Image set of FDS-1]{
        \includegraphics[width=0.47\textwidth]{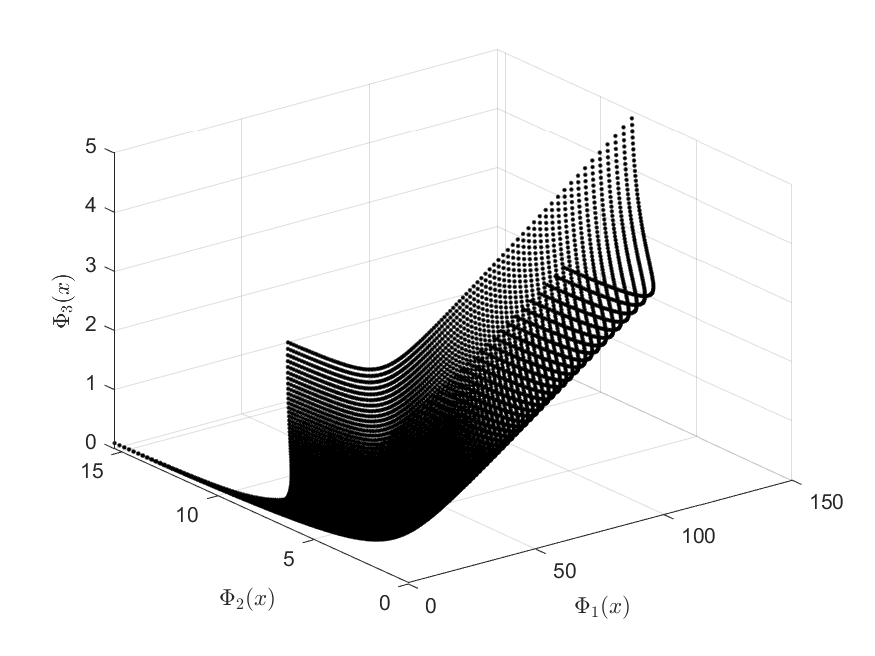}
    }
    \subfigure[FDS-1] {
        \includegraphics[width=0.47\textwidth]{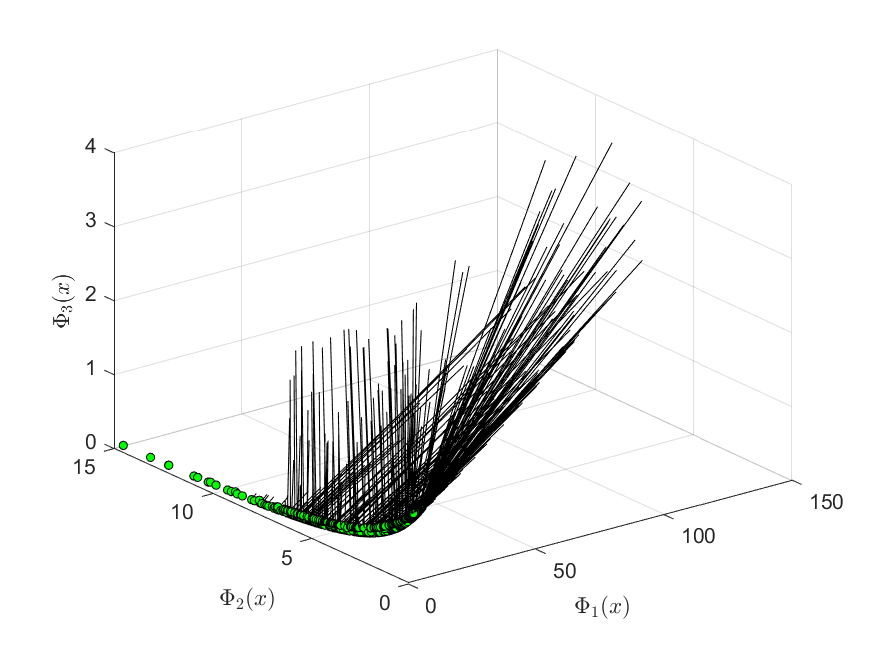}
    }\\[1mm]

    \subfigure[Image set of MOP5]{
        \includegraphics[width=0.47\textwidth]{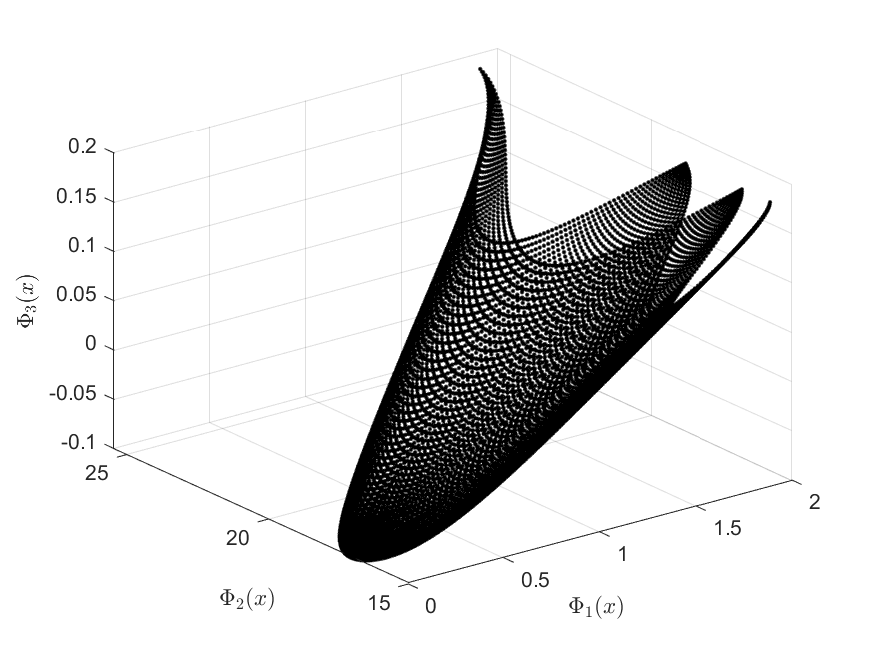}
    }
    \subfigure[MOP5]{
        \includegraphics[width=0.47\textwidth]{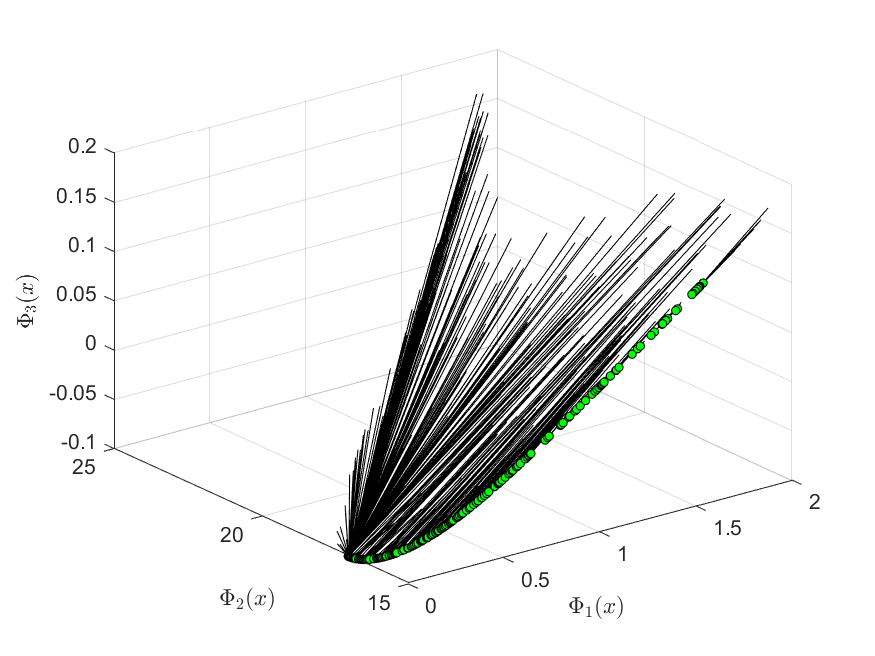}
    }\\[1mm]

    \subfigure[Image set of MOP7]{
        \includegraphics[width=0.47\textwidth]{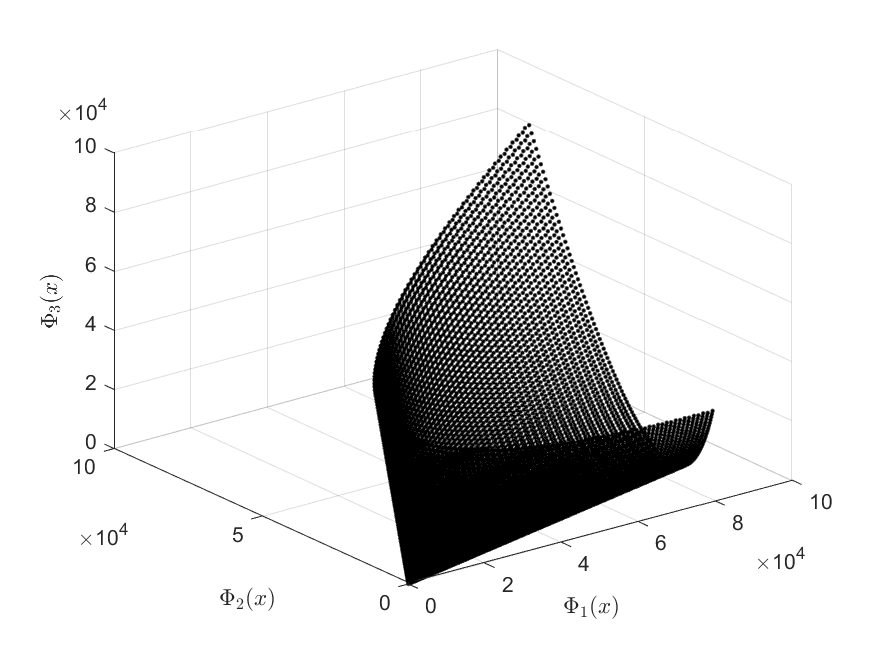}
    }
    \subfigure[MOP7]{
        \includegraphics[width=0.47\textwidth]{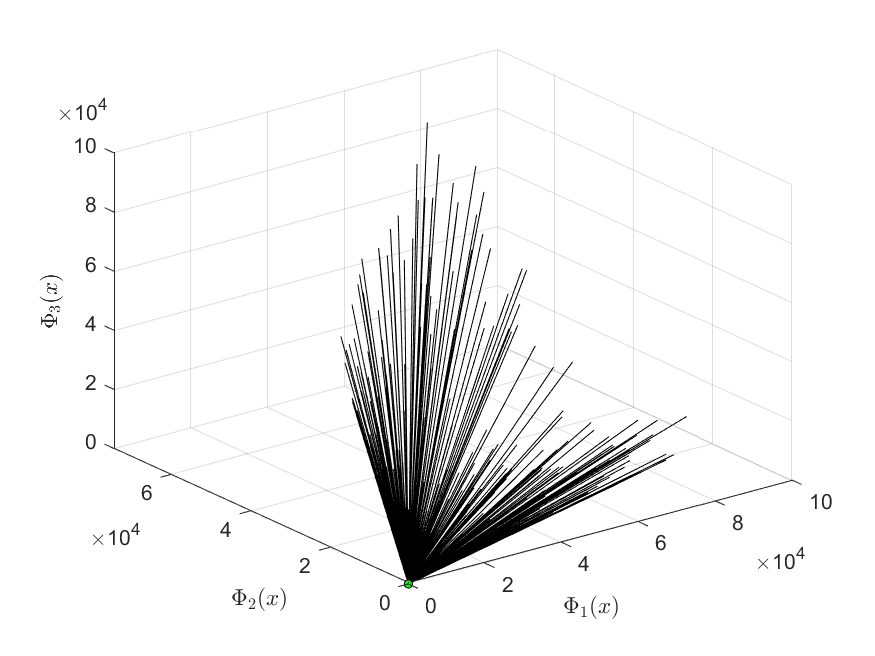}
    }

  \caption{Approximations of  Pareto frontiers generated by the TT-PRP   method and  corresponding  image sets   for three-criteria problems. }
\end{figure}

\section{Final remarks and future directions}

We  propose a novel  three-term PRP CG method for  vector optimization, which can be regarded  as the first extension of  three-term CG methods to the field of vector optimization.  This vector  extension retains the main advantage of classical three-term CG methods, i.e.,  the generated search direction  is always a sufficient descent direction regardless of the line search and without the modification of conjugate parameters.   This serves as a significant supplement to the PRP CG method for vector optimization, which, as mentioned in Example \ref{e1}, may not independently generate descent directions. In addition, based on a new Wolfe-type line search, we establish the global convergence of the
proposed method without imposing restrictive conditions. Numerical experiments show that the proposed scheme has better performance than the PRP vector extension.
However,  the convergence conclusion in this work  is limited to the Wolfe-type line search. The corresponding convergence result  under more relaxed conditions, such as the Armijo-type line search,  has not yet been investigated.
So it would be  interesting to explore this problem and consider the corresponding numerical performance.

Since the  vector extension of CG methods  may lose  the descent property of search directions, it would be worthwhile to explore the vector extension  of other three-term CG method  while retaining  its great properties as in scalar optimization.
Additionally, as far as we know, developing a general form of three-term CG methods for vector optimization remains an open problem, which is the key direction of our future research.

\section*{Declarations}
The authors declare  that they have no competing interests.

\end{document}